\documentclass[11pt]{amsart}

\usepackage{hyperref}
\makeatletter
\def\@secnumfont{\bfseries}
\makeatletter

\hypersetup{
	colorlinks=true,
	linkcolor=blue,
	citecolor=cyan,
	urlcolor=cyan}
	
	\usepackage{hyphenat}
	\usepackage{epsfig, subfig, amsmath, amssymb, latexsym,amscd, amsfonts, xcolor, easybmat, bbm}

\newcommand{\bh}{\mathcal{B}(\mathcal{H})}
\newcommand{\dist}{\text{dist}}

\newcommand{\ee}{\varepsilon}

\newcommand{\id}{\text{id}}
\newcommand{\into}{\hookrightarrow}

\newcommand{\I}{\text{I}}
\newcommand{\II}{\text{II}}
\newcommand{\III}{\text{III}}

\newcommand{\kpn}{$[\text{KP}_n]$\, }

\newcommand{\MvN}{\sim_{\text{MvN}}}

\newcommand{\ov}{\overline{\otimes}}

\newcommand{\rank}{\text{rank}}

\newcommand{\tr}{\text{tr}}

\DeclareMathOperator{\Aut}{Aut}
\DeclareMathOperator{\End}{End}
\DeclareMathOperator{\ind}{ind}

\newenvironment{cor}{\subsection{}{\textbf {Corollary.}}\em}{}
\newenvironment{defn}{\subsection{}{\textbf {Definition.}}\em}{\smallskip}

\newenvironment{lemma}{\subsection{}{\textbf {Lemma.}}\em}{\smallskip}

\newenvironment{prop}{\subsection{}{\textbf {Proposition.}}\em}{\smallskip}

\newenvironment{remark}{\subsection{}{\textbf {Remark.}}}{\smallskip}
\newenvironment{theorem}{\subsection{}{\textbf {Theorem.}}\em}{\smallskip}

\newenvironment{ques}{\subsection{}{\textbf{Question.}}}{\smallskip}

\newcommand\cB{\ensuremath{\mathcal B}}

\newcommand\cD{\ensuremath{\mathcal D}}

\newcommand\cH{\ensuremath{\mathcal H}}

\newcommand\cM{\ensuremath{\mathcal M}}

\newcommand\cO{\ensuremath{\mathcal O}}

\newcommand\cQ{\ensuremath{\mathcal Q}}
\newcommand\cR{\ensuremath{\mathcal R}}

\newcommand\cZ{\ensuremath{\mathcal Z}}

\newcommand\bA{\ensuremath{\mathbb A}}
\newcommand\bB{\ensuremath{\mathbb B}}
\newcommand\bC{\ensuremath{\mathbb C}}

\newcommand\bI{\ensuremath{\mathbb I}}

\newcommand\bK{\ensuremath{\mathbb K}}

\newcommand\bM{\ensuremath{\mathbb M}}
\newcommand\bN{\ensuremath{\mathbb N}}

\newcommand\bR{\ensuremath{\mathbb R}}

\newcommand\bT{\ensuremath{\mathbb T}}

\newcommand\bZ{\ensuremath{\mathbb Z}}

\newenvironment{question}{\subsection{}{\textbf{Question.}}}{\smallskip}

\newcommand{\etalchar}[1]{$^{#1}$}
\providecommand{\bysame}{\leavevmode\hbox to3em{\hrulefill}\thinspace}
\providecommand{\MR}{\relax\ifhmode\unskip\space\fi MR }

\providecommand{\href}[2]{#2}

\begin{document}


\title{Kaplansky's problem and unitary orbits in matrix amplifications}


\thanks{${}^1$ Research supported in part by NSERC (Canada)}
\thanks{${}^2$ Research supported in part by National Natural Science Foundation of China (No.: 12471123).}

\thanks{{\ifcase\month\or Jan.\or Feb.\or March\or April\or May\or
June\or
July\or Aug.\or Sept.\or Oct.\or Nov.\or Dec.\fi\space \number\day,
\number\year}}
\author
	[Laurent W. Marcoux]{{Laurent W.~Marcoux${}^1$}}
\address
	{Department of Pure Mathematics\\
	University of Waterloo\\
	Waterloo, Ontario \\
	Canada  \ \ \ N2L 3G1}
\email{\href{mailto:Laurent.Marcoux@uwaterloo.ca}{Laurent.Marcoux@uwaterloo.ca}}

%
%
%

\author
	[Pawe{\l} Sarkowicz]{{Pawe{\l} Sarkowicz}}
\address
	{Department of Pure Mathematics\\
	University of Waterloo\\
	Waterloo, Ontario \\
	Canada  \ \ \ N2L 3G1}
\email{\href{mailto:psarkowi@uwaterloo.ca}{psarkowi@uwaterloo.ca}}


\author
	[Y.H.~Zhang]{{Yuanhang~Zhang${}^2$}}
\address
	{School of Mathematics\\
	Jilin University\\
	Changchun 130012\\
	P.R. CHINA}
\email{\href{mailto:zhangyuanhang@jlu.edu.cn}{zhangyuanhang@jlu.edu.cn}}


\begin{abstract}
  We study the distances between the unitary orbits of matrix amplifications of elements in certain C*-algebras.
  In particular, we show that the distance between unitary orbits of arbitrary elements in unital, separable, UHF-stable C*-algebras  remains unchanged when amplifying to certain matrix sizes.
  We further exhibit examples of elements in C*-algebras where the distance between unitary orbits becomes strictly smaller after amplifying by a certain matrix size, and we demonstrate that  distances between unitary orbits of amplifications are not monotone in the multiplicity of the amplifications, even in the setting of matrix algebras. Lastly, we show that topological $K$-theory provides obstructions in the purely infinite setting.

\end{abstract}



\maketitle
\tableofcontents
\markboth{\textsc{  }}{\textsc{}}



  \addtocontents{toc}{\protect\setcounter{tocdepth}{1}}

\section{Introduction}\label{sec:intro}

In Kaplansky's book on infinite abelian groups, three \emph{test problems} were posed for abelian groups \cite[\textsection 6]{Kaplansky54}, the second being of particular interest here. The question is the following: if $G \oplus G$ and $H \oplus H$ are isomorphic as abelian groups, are $G$ and $H$ isomorphic?  (For finite abelian groups and divisible groups, the answer is yes, while J\'{o}nsson~\cite{Jonsson57} and Corner~\cite{Corner64} have provided counterexamples amongst infinite abelian groups.)
By replacing isomorphism of groups by other notions of equivalence of objects in other categories, one obtains natural analogues of this question in various areas of mathematics, including operator theory and operator algebras.
For example, Kadison and Singer showed in \cite{KadisonSinger57} that if $S,T \in \bh$ are bounded linear operators acting on a Hilbert space $\cH$, then
\begin{equation}
\begin{pmatrix}
S & 0\\ 0& S
\end{pmatrix} \text{ and } \begin{pmatrix}
T & 0\\ 0& T
\end{pmatrix}
\end{equation}
are unitarily equivalent in $\bM_2(\bh)$ if and only if $S$ and $T$ are unitarily equivalent in $\bh$.
That is, if the multiplicity-two  amplifications of $S$ and $T$ lie in the same unitary orbit, then the operators themselves share a unitary orbit.
The corresponding question also has a positive answer in general von Neumann algebras by results of Azoff \cite{Azoff95} (see also Sherman \cite[Appendix A.2]{Sherman10}).
It is worth noting that Kaplansky's problem,  with unitary equivalence replaced by similarity, is still open, even for $\bh$. A variant of this problem was  considered in \cite{MRTZ24}, where the authors proved that similarity of so-called \emph{primitive square roots} of  amplifications of compact operators $K$ and $L$  imply similarity of $K$ and $L$ themselves.

Unitary orbits of elements in unital C*-algebras are not closed in general,  indeed, not even in $\bh$\footnote{Take a non-diagonalizable normal operator $N \in \bh$. A corollary of the Weyl-von Neumann-Berg theorem (see \cite[Corollary II.4.5]{DavidsonBook1}) yields that $N$ is approximately unitarily equivalent to a diagonal operator, say $M$. That is, $M$ lies in the closure of the unitary orbit of $N$, but not in the unitary orbit itself.
In fact, when $\cM$ is any factor, it is known precisely when the unitary orbit of a normal operator is norm closed -- see \cite[Theorem 8.13]{Sherman07Unitary}.
One can also construct two positive elements in the CAR algebra which are approximately unitarily equivalent but not unitarily equivalent \cite[Example 2.2]{MarcouxZhang21}}, so it is natural to consider the question of whether
\begin{equation}
\begin{pmatrix}
S & 0\\ 0& S
\end{pmatrix} \text{ and } \begin{pmatrix}
T & 0\\ 0& T
\end{pmatrix}
\end{equation}
being \emph{approximately} unitarily equivalent in $\bM_2(\bh)$ implies that $S$ and $T$ are approximately unitarily equivalent in $\bh$.
This question was answered positively in \cite[Proposition 4.4]{MRTZ24}. The problem remains both interesting and open if one generalises from $\bh$ to arbitrary unital C*-algebras.

\begin{ques} \ \textbf{[Kaplansky's Problem for approximate unitary equivalence]}\label{Kaplansky's problem}

Let $\bA$ be a unital C*-algebra and $a,b \in \bA$. If $a \oplus a$ is approximately unitarily equivalent to $b \oplus b$ in $\bM_2(\bA)$, does it follow that $a$ is approximately unitarily equivalent to $b$?  More generally, if $n \ge 1$ and the $n^{\text{th}}$-matrix amplifications $a^{(n)}$ and $b^{(n)}$ are approximately unitarily equivalent in $\bM_n(\bA)$, are $a$ and $b$ approximately unitarily equivalent in $\bA$?
\end{ques}

When this holds for a fixed $n \in \bN$, we shall say that $a$ \emph{and} $b$ \emph{satisfy} \kpn (\emph{for approximate unitary equivalence in} $\bA$).

\smallskip

One can further formulate a quantitative version of Kaplansky's problem as follows.

\begin{ques} \ \textbf{[The distance problem for unitary orbits]} \label{Kaplansky's problem unitary orbits}

Let $\bA$ be a unital C*-algebra and $n \ge 2$ be an integer.  When is it the case that the distance between the unitary orbits of $a^{(n)}$ and $b^{(n)}$ in $\bM_n(\bA)$ is equal to the distance between the unitary orbits of $a$ and $b$ in $\bA$?
\end{ques}

Observe that when this is indeed the case, we obtain a positive answer to Kaplansky's Question~\ref{Kaplansky's problem}.  The fact that every unitary $u$ element of $\bM_n(\bA)$ can be extended to a unitary element  $\bM_{n+1}(\bA)$ in many ways -- e.g. $u \oplus v$ for any unitary element $v$ of $\bA$ -- suggests that the distance between the unitary orbits of $a^{(n)}$ and $b^{(n)}$ may be monotonically decreasing as a function of $n$.


\smallskip

\begin{ques} \ \textbf{[The monotonicity problem for distances between unitary orbits of amplifications]} \label{factor-pair-amplification}

Let $\bA$ be a unital C*-algebra. Is it the case that
\begin{equation}
\dist(U(a^{(k)}),U(b^{(k)})) \leq \dist(U(a^{(l)}),U(b^{(l)}))
\end{equation}
whenever $k \geq l$?
\end{ques}

As we shall see, the answers to these questions depend upon the underlying C*-algebras, the class of operators under consideration (e.g. self-adjoint elements, normal elements, etc.), as well as  the multiplicity of the amplification.





\section{Preliminaries and notation}\label{sec:prelims}


\subsection{Notation}\label{subs:notation}
For a unital C*-algebra $\bA$, $\bA_{sa}$ will denote the set of self-adjoint elements of $\bA$, while $U(\bA)$ will denote the unitary group of $\bA$; that is, the subgroup of invertible elements $u \in \bA$ which satisfy
\begin{equation}
u^*u = uu^* = 1_{\bA}.
\end{equation}
We will write $U^0(\bA)$ for the connected component of the unitary group, i.e., the subgroup of unitaries which are path-connected to the unit of $\bA$.
For an element $a \in \bA$, we shall write $U(a)$ for the \emph{unitary orbit} of $a$, i.e.,
\begin{equation}
U(a) := \{u^*au \mid u \in U(\bA)\}.
\end{equation}

Let $n \in \bN$. For brevity, we will denote by $\bM_n$ the algebra of $n \times n$ complex matrices, while $\bM_n(\bA)$ will denote the matrix amplification of a C*-algebra $\bA$. We will often identify $\bM_n(\bA)$ canonically with the tensor product $\bM_n \otimes \bA$.  That is, given $a \in \bA$, and using the standard basis $\{ e_{ij}\}_{1 \le i, j \le n}$ for $\bM_n$, the element $e_{ij} \otimes a$ corresponds to a matrix in $\bM_n(\bA)$ with $a$ in the $(i,j)$-entry and zeros elsewhere.   For an element $a \in \bA$, we will denote by $a^{(n)}$ the element $1_{\bM_n} \otimes a \in \bM_n \otimes \bA$, or equivalently, the $n$-fold direct sum of $a$ with itself in $\bM_n(\bA)$.
We will use
\[
\dist(U(a),U(b)) := \inf \{ \| x - y\| : x \in U(a), y \in U(b) \} \]
to denote the distance between the unitary orbits of $a$ and $b$. It will be clear from context where the unitaries lie: if $a,b \in \bA$ then $\dist(U(a),U(b))$ will be computed using unitary elements from $\bA$, while $\dist(U(a^{(n)}),U(b^{(n)}))$ will be computed using unitary elements coming from $\bM_n(\bA)$.

We will denote by $\bA \otimes \bB$ the minimal tensor product of the C*-algebras $\bA$ and $\bB$.
Usually, one of the C*-algebras in question will be nuclear, so there will be no ambiguity.
When $\bA = \bM_n$, this reduces to the algebra $\bM_n(\bB)$ mentioned above.

If $\varphi: \bA \to \bB$ is a (usually unital) *-homomorphism, we will denote by $\varphi^{(n)}: \bA \to \bM_n(\bB)$ the *-homomorphism $\varphi^{(n)} =  1_{\bM_n} \otimes \varphi$, which can be identified with the map
\begin{equation}
\varphi^{(n)}(a) = \begin{pmatrix}
\varphi(a) \\ & \ddots \\ & & \varphi(a)
\end{pmatrix} \in \bM_n(\bB).
\end{equation}
We note that if $\tau$ is a tracial state on $\bA$, then $\tau$ can be extended canonically to a tracial state $\tau_n$ on $\bM_n(\bA) = \bM_n \otimes \bA$ by taking $\tau_n((a_{ij})) = \frac{1}{n}\sum_i \tau(a_{ii})$.
This is nothing more than the map $\frac{1}{n}\tr_n \otimes \tau:  \bM_n \otimes \bA \to \bC$, where $\tr_n: \bM_n \to \bC$ is the usual unnormalized trace.

If $\bA$ is a unital C*-algebra, we will write $a \simeq b$ to mean that $a$ and $b$ are \emph{unitarily equivalent} in $\bA$; that is, there is a unitary $u \in U(\bA)$ such that $u^*au = b$.
We will write $a \simeq_a b$ to mean that $a$ and $b$ are \emph{approximately unitarily equivalent}, which is to say that there exists a sequence of unitaries $(u_n)_n$ in $U(\bA)$ such that $u_n^*au_n \to b$ in norm.
This is an equivalence relation on $\bA$, and $a \simeq_a b$  if and only if $\dist(U(a),U(b)) = 0$, or equivalently, if $U(a)$ and $U(b)$ have the same norm-closures.
Lastly, if $\bB$ is another unital C*-algebra and $\varphi,\psi: \bA \to \bB$ are two *-homomorphisms, we say that $\varphi$ and $\psi$ are \emph{approximately unitarily equivalent}, written $\varphi \simeq_a \psi$, if there exists a sequence $(u_n)_n$ of  unitary elements of  $U(\bB)$ such that $u_n^*\varphi(a)u_n \to \psi(a)$ for all $a \in \bA$.

Two projections $p$ and $q$ in a C*-algebra $\bA$ are said to be \emph{Murray-von Neumann equivalent}, in which case we write $p \sim_{\text{MvN}} q$, if there exists a partial isometry $v \in \bA$ such that $p = v^* v$ and $q = v v^*$.  The $K_0$-group of a unital C*-algebra $\bA$ is the Grothendieck group of the semigroup of Murray-von Neumann equivalence classes of projections in matrix algebras over $\bA$, with addition being the direct sum.
When $\tau$ is a tracial state on $\bA$, there is an associated pairing map
\begin{equation}
\tau_*: K_0(\bA) \to \bR
\end{equation}
given by $\tau_*([p] - [q]) = \tr_n \otimes \tau(p - q)$ whenever $p,q \in \bM_n(\bA) =\bM_n \otimes \bA$ are projections. The $K_1$-group $K_1(\bA)$ of a unital C*-algebra $\bA$ is the group of homotopy equivalences classes of unitaries in matrix amplifications, with addition being direct sum.

We shall  denote by $\cZ$ the \emph{Jiang-Su algebra}~\cite{JiangSu99}.  We recall that $\cZ$   is a unital, simple, separable, nuclear C*-algebra satisfying the Universal Coefficient Theorem (UCT) of Rosenberg and Schochet \cite{RosenbergSchochet87}.  It is $KK$-equivalent to $\bC$, and is tensorially self-absorbing. It is often thought of as the infinite-dimensional analogue of the complex numbers from a K-theoretic point of view. From the point of view of classification theory of amenable C*-algebras via $K$-theory and traces, for a given C*-algebra $\bA$, the property of being $\cZ$-stable -- that is, $\bA \otimes \cZ \simeq \bA$ -- is a desirable regularity property and a crucial hypothesis.


\section{Self-adjoint and normal elements in certain C*-algebras}\label{sec: pos results from lit}

One of the main difficulties which arises in addressing Question~\ref{Kaplansky's problem unitary orbits} is that in general, it is very difficult (and in some cases currently impossible) to precisely compute the distance between unitary orbits of elements of C*-algebras.  One class of elements for which such computations are often possible is the set of self-adjoint elements, though even here, our ability to compute the distances depends greatly upon the underlying C*-algebra.      We shall later demonstrate that the distance between unitary orbits of amplifications of two projections $p$ and $q$ in a C*-algebra $\bA$ can differ from the distance between the unitary orbits of $p$ and $q$ themselves (cf. Proposition~\ref{prop: non n-KP}).
But first let us consider some cases where the distance between unitary orbits of self-adjoint elements remains unchanged under amplification.


\subsection{Self-adjoint matrices}\label{subs:self-adjoint matrices}

We begin by giving a simple proof of the fact  that the distance between unitary orbits of self-adjoint matrices remains unchanged upon matrix amplification, thereby providing a positive answer to Question~\ref{Kaplansky's problem unitary orbits} in this setting.
For two self-adjoint matrices $S,T \in \bM_k$, we define the \emph{spectral distance} between them to be  the optimal matching distance between their spectra:
\begin{equation}\label{eq:spectral distance}
\text{sprd}(S, T) := \min_{\sigma \in \mathfrak{S}_k} \max_{1 \leq j \leq k} |x_j - y_{\sigma(j)}|,
\end{equation}
where $\mathfrak{S}_k$ is the permutation group on $\{ 1, 2, \ldots, k\}$ and $(x_j)$ and $(y_j)$ are the eigenvalues of $S$ and $T$ counted according to algebraic multiplicity, respectively.
It was proven by Weyl in \cite{Weyl12} that the spectral distance between self-adjoint matrices is equal to the distance between their unitary orbits in $\bM_k$ (see also \cite[Theorem 1.2]{AzoffDavis84}).  Let us provide a  relatively straightforward computation of the spectral distance.


\begin{lemma}
Let $x_1,x_2,y_1,y_2 \in \bR$ be such that $x_1 \leq x_2$ and $y_1 \leq y_2$. Then
\begin{equation}\label{eq:diamter pairs of points}
\max\{|y_1 - x_1|,|y_2 - x_2|\} \leq \max\{|y_1 - x_2|,|y_2 - x_1|\}.
\end{equation}
\begin{proof}
By interchanging the $x_i$'s and $y_i$'s if necessary, we can assume that $y_1 \leq x_1$.
There are two cases two consider, namely:  when $y_2 \leq x_2$ and when $y_2 > x_2$.

In the former case, $y_1 \leq y_2 \leq x_2$, and so the largest distance between any two points in $\{x_1,x_2,y_1,y_2\}$ is $x_2 - y_1 = |x_2 - y_1|$, so that (\ref{eq:diamter pairs of points}) holds.

For the latter case,
\begin{equation}
\begin{split}
|y_2 - x_1| &= y_2 - x_1 \\
&\geq y_2 - x_2 \\
&= |y_2 - x_2|
\end{split}
\end{equation}
and similarly $|x_2 - y_1| \geq |x_1 - y_1|$. That is, (\ref{eq:diamter pairs of points}) holds.
\end{proof}
\end{lemma}


\begin{prop}
Let $x_1 \leq x_2 \leq \cdots \leq x_k$ and $y_1 \leq y_2 \leq \cdots \leq y_k$ be elements of $\bR$. If $\sigma \in \mathfrak{S}_k$, then
\begin{equation}
\max_{1 \leq j \leq k}|x_j - y_j| \leq \max_{1 \leq j \leq k} |x_j - y_{\sigma(j)}|.
\end{equation}
\begin{proof}
Let $z_j = y_{\sigma(j)}$ and suppose that there exist $1 \leq i_0 < j_0 \leq n$ such that $z_{i_0} > z_{j_0}$.

Let $\rho \in \mathfrak{S}_k$ be the permutation $(i_0,j_0)$. Note that if $m \notin \{i_0,j_0\}$, then
\begin{equation}
|x_m - z_m| = |x_m - z_{\rho(m)}|.
\end{equation}
By the above lemma, we have that
\begin{equation}
\max_{m \in \{i_0,j_0\}} |x_m - z_{\rho(m)}| \leq \max_{m \in \{i_0,j_0\}} |x_m - z_m|.
\end{equation}
In other words, by ordering $z_{i_0},z_{j_0}$ in increasing order, our pairing distance improves.
We can therefore apply a finite number of transpositions in order to eventually rewrite $y_1,\dots,y_k$ in increasing order, and at the last stage, the pairing distance can no longer  decrease. That is, the result holds.
\end{proof}
\end{prop}


\begin{cor}
Let $n \in \bN$ and $S,T \in \bM_k$ be self-adjoint matrices with
\begin{equation}
\begin{split}
\sigma(S) &= \{x_1,\dots,x_k\} \\
\sigma(T) &= \{y_1,\dots,y_k\}.
\end{split}
\end{equation}
Then
\begin{equation}
\text{sprd}(S^{(n)},T^{(n)}) = \text{sprd}(S,T).
\end{equation}
\begin{proof}
By the above proposition, we have that
\begin{equation}
\text{sprd}(S,T) = \max_{1 \leq j \leq k}|x_j - y_j|,
\end{equation}
where $x_1 \leq x_2 \leq \cdots \leq x_k$ and $y_1 \leq y_2 \leq \cdots \leq y_k$ are written in increasing order.
Writing $\sigma(S^{(n)}), \sigma(T^{(n)})$ in increasing order obviously yields the same result.
\end{proof}
\end{cor}


As mentioned above, by \cite[Theorem 1.2]{AzoffDavis84}, the spectral distance between two self-adjoint matrices is equal to the distance between their unitary orbits. We note that Holbrook produced in \cite{Holbrook92} an example of two $3\times3$ normal matrices whose spectral distance is strictly greater than the distance between their unitary orbits.  

This fact about self-adjoint matrices allows us to conclude the following.


\begin{cor}
Let $n \in \bN$ and $S,T \in \bM_k$ be self-adjoint matrices. Then
\begin{equation}
\dist(U(S),U(T)) = \dist(U(S^{(n)}),U(T^{(n)})).
\end{equation}
\end{cor}


We note this is not true for arbitrary pairs of matrices.  Indeed, \cite[Theorem 5.1]{MarcouxZhang21}   proves the existence of $n,k \in \bN$ and $a,b \in \bM_k$ such that
\begin{equation}
\dist(U(a^{(n)}),U(b^{(n)})) < \dist(U(a),U(b)).
\end{equation}
(We note, however, that  Theorem~5.3 of the same paper shows that the conclusion of the above Corollary holds for all $k \ge 1$ when considering pairs $S, T$ of \emph{normal} matrices in $\bM_2$.)
The question of finding explicit matrices for which the strict inequality holds remains.
We can, however, provide a sharper idea of the integers that one can work with. In fact, one only needs to consider matrices of size $2^n \times 2^n$ (although the precise value of $n$ required is still unknown), and to look at the 3-fold amplification (see Proposition \ref{prop: existence of non-KP matrices}).


\subsection{Self-adjoint elements in real rank zero C*-algebras with (FCQ1)}\label{subs:self-adjoint in real rank zero}

Recall that a C*-algebra $\bA$ has \emph{real rank zero} if the set of self-adjoint elements with finite spectrum  is dense in the set of all self-adjoint elements of $\bA$.
Heuristically, this means that the C*-algebra (and its hereditary subalgebras) have a plethora of projections.  A unital C*-algebra with faithul trace $\tau$ satisfies Blackadar's \emph{Fundamental Comparability Question 1} (FCQ1)~\cite{Blackadar88} with respect to $\tau$ if, whenever $p,q \in \bA$ are projections such that $\tau(p) \leq \tau(q)$, we have $p \lesssim q$ (i.e., $p$ is Murray-von Neumann equivalent to a subprojection of $q$).  This has also been referred to as \emph{strong comparison of projections} in \cite{Skoufranis16}.
There are many C*-algebras which satisfy (FCQ1), as pointed out in \cite[Example 3.8]{Skoufranis16}.

In that paper, Skoufranis computes a formula for the distance between the unitary orbits of any two self-adjoint elements in a C*-algebra of real rank zero.   His formula allows us to show that the distance is invariant under matrix amplification.

If $\bA$ is a unital C*-algebra equipped with a tracial state $\tau$, we  define the \emph{dimension function} associated to $\tau$ by
\begin{equation}
d_\tau(a) = \lim_{\ee \to 0^+} \tau(f_\ee(|a|))
\end{equation}
where $f: \bR \to \bR$ is the piece-wise linear function such that
\begin{equation}
f_\ee(t) = \begin{cases}
0 & t \leq \frac{\ee}{2} \\
\frac{2}{\ee} (t - \frac{\ee}{2}) & t \in [\frac{\ee}{2},\ee] \\
1 & t \geq \ee.
\end{cases}
\end{equation}
The dimension function $d_\tau$ just gives the trace of the support projection of $a$ in the double dual of $\bA$, cf. \cite[Remark 3.5]{OrtegaRordamThiel11}. For example, for a matrix $a \in \bM_n$, we have that $d_\tau(a) = \frac{\rank(a)}{n}$, where $\tau: \bM_n \to \bC$ is the unique tracial state.

For a self-adjoint element $b$ in a C*-algebra $\bA$, we shall write  $b_+$ to denote the positive part of $b$ obtained via functional calculus.


\begin{defn}
Let $\bA$ be a unital C*-algebra with faithful tracial state $\tau$.
We define, for $a \in \bA_{sa}$, the \emph{eigenvalue function of $a$ associated to $\tau$}, which we denote by $\lambda_a^\tau$, to be the function $[0,1) \to \bR$ defined by
\begin{equation}
\lambda_a^\tau(s) := \inf \{t \in \bR \mid d_\tau((a - t1_\bA)_+) \leq s\}.
\end{equation}
\end{defn}


The eigenvalue function is a well-defined and well-understood function \cite[Theorem 2.10]{Skoufranis16}.
In the case of self-adjoint matrices, if one orders the eigenvalues of $a \in \bM_k$ as $\lambda_1 \leq \lambda_2 \leq \cdots \leq \lambda_k$, then the eigenvalue function is the step-function given by
\begin{equation}
\lambda_a^\tau(s) = \lambda_{k-j+1} \text{ whenever } s \in \left[\frac{j-1}{k},\frac{j}{k}\right).
\end{equation}


The following is \cite[Theorem 5.1]{Skoufranis16}, and can also be found in a different form in \cite{SunderThomsen92}.


\begin{theorem}\label{thm: Skoufranis distance in terms of eigenfunctions}
Let $\bA$ be a unital C*-algebra with real rank zero which satisfies FCQ1 with respect to a faithful tracial state $\tau$. If $a,b \in \bA_{sa}$, then
\begin{equation}
\dist(U(a),U(b)) = \sup_{s \in [0,1)} |\lambda_a^\tau(s) - \lambda_b^\tau(s)|.
\end{equation}
\end{theorem}


We are now ready to apply these results to the amplification problem in certain unital simple C*-algebras of real rank zero.  We remind the reader that a unital C*-algebra $\bA$ is said to be \emph{finite} if it contains no infinite projections (i.e. projections Murray-von Neumann equivalent to a proper subprojection of themselves), and $\bA$ is said to be  \emph{stably finite} if $\bM_n(\bA)$ is finite for all $n \ge 1$~\cite{Rordam2004}.  We also note that if $\bA$ is simple, then any tracial state on $\bA$ is automatically faithful.

\begin{prop}\label{prop: KPn for self-adjoints in nice C*-algebras}
Let $\bA$ be a unital, simple, stably finite C*-algebra of real rank zero with tracial state $\tau$.   Let $\tr_n$ denote the standard unnormalised trace on $\bM_n$, and suppose that  $\bM_n(\bA)$ satisfies Blackadar's FCQ1 with respect to $\tau_n := \frac{1}{n}\tr_n \otimes \tau$ for all $n$. Then
\begin{equation}\label{eq: dist for self-adjoints}
\dist(U(a),U(b)) = \dist(U(a^{(n)}),U(b^{(n)}))
\end{equation}
for all $a,b \in \bA_{sa}$ and $n \in \bN$.

\begin{proof}
Fix $n \in \bN$.  Then $\bM_n(\bA)$ is a unital, simple, stably finite C*-algebra of real rank zero, and it has faithful trace $\tau_n$ as defined above.
We have by Theorem \ref{thm: Skoufranis distance in terms of eigenfunctions} that
\begin{equation}\label{eq: eigenvalue-1}
\dist(U(a),U(b)) = \sup_{s \in [0,1)} |\lambda_a^\tau(s) - \lambda_b^\tau(s)|
\end{equation}
and
\begin{equation}\label{eq:eigenvalue-n}
\dist(U(a^{(n)}),U(b^{(n)})) = \sup_{s \in [0,1)} |\lambda_{a^{(n)}}^{\tau_n}(s) - \lambda_{b^{(n)}}^{\tau_n}(s)|.
\end{equation}
To see the quantities on the right of (\ref{eq: eigenvalue-1}) and (\ref{eq:eigenvalue-n}) are equal, notice that
\begin{equation}
(a^{(n)} - t1_{\bM_n(\bA)})_+ = (a - t1_\bA)_+^{(n)}, \ t \in \bR,
\end{equation}
and that
\begin{equation}
d_{\tau_n}((a - t1_\bA)_+^{(n)}) = d_\tau((a - t1_\bA)_+)
\end{equation}
which yields that $\lambda_{a^{(n)}}^{\tau_n}(s) = \lambda_a^\tau(s)$ for all $s \in [0,1)$, and similarly for $b$.
\end{proof}
\end{prop}


It is well-known that both (matrix amplifications of) UHF algebras as well as $\II_1$ factors satisfy FCQ1 with respect to their unique trace.
We will see that the collection of algebras satisfying (FCQ1) contains many C*-algebras that arise in the classification programme.

There is a second comparison property for projections, known alternately as the FCQ2 property or  \emph{strict comparison of projections} in the literature.  To be precise, a simple, unital C*-algebra $\bA$  is said to \emph{satisfy} FCQ2 if, for all projections $p, q \in \bA$, $\tau(p) < \tau(q)$ for all tracial states $\tau$ implies that $p \prec q$ (or equivalently $p \preceq q$); that is, $p$ is Murray-von Neumann equivalent to a  subprojection of $q$.
This is a fundamental regularity property that is exhibited by many C*-algebras, for example by those satisfying the hypotheses of the classification theorem (see for example \cite{Rordam04}), as well as those coming from free product C*-algebras \cite{DykemaRordam98,DykemaRordam00}. More generally, many of these C*-algebras satisfy \emph{strict comparison of positive elements}, where one considers Cuntz equivalence of positive elements \cite{Cuntz77structmult,Cuntz78}, which plays a role in the classification programme at large and is related to the Toms-Winter conjecture. Recently, several exact C*-algebras coming from free products and groups were shown to satisfy strict comparison of positive elements \cite{AGEKP24,HayesKunnElaRobert25}

Despite FCQ1 being common among real rank zero C*-algebras, it is far less common than strict comparison of projections; in fact, there are plenty of simple AF algebras without this property. The point is that $\tau(p) = \tau(q)$ need not imply that $p \MvN q$ in general. Indeed, there are even unital, simple AF algebras with unique trace where this does not hold -- see \cite[Section 1.5]{RordamBook}, and \cite[Section 7.6]{BlackadarKBook}, for example. We note that these AF algebras satisfy all the assumptions of Lemma~\ref{lem3.08} below -- i.e., they are unital, simple, separable, $\cZ$-stable C*-algebras which are of real rank zero -- except that the pairing map from $K_0(\bA)$ to $\bR$ will not be injective in general.


Recall that a unital C*-algebra $\bA$ has \emph{stable rank one} if the invertible elements are dense in $\bA$. It is well-known that stable rank one C*-algebras are stably finite, cf. \cite[Proposition 3.3.4]{LinBook}.
We first require a lemma.


\begin{lemma} \label{lem3.08}
Let $\bA$ be a unital, simple C*-algebra of stable rank one with tracial state $\tau$. Suppose that the pairing map $\tau_*: K_0(\bA) \to \bR$ is injective and that $\bA$ satisfies strict comparison of projections. Then $\bA$ satisfies FCQ1 with respect to $\tau$.
\begin{proof}
Suppose that $p,q \in \bA$ are projections and that $\tau(p) \leq \tau(q)$.

The case where $\tau(p) < \tau(q)$ follows from strict comparison.

If $\tau(p) = \tau(q)$, then $[p]_0 = [q]_0$ by the injectivity of $\tau_*$. As $\bA$ has stable rank one, it also has cancellation, by \cite[6.5.1]{BlackadarKBook}.  From this it follows that $p \simeq q$, which clearly implies that $p \preceq q$.
\end{proof}
\end{lemma}


The above will allow us to additionally conclude that many algebras satisfy (\ref{eq: dist for self-adjoints}) when we restrict to the self-adjoint elements.


\begin{cor}
Let $\bA$ be a unital, simple, separable, $\cZ$-stable C*-algebra of real rank zero with unique tracial state $\tau$ such that the pairing map $\tau_*: K_0(\bA) \to \bR$ is injective.
Then, whenever $a,b \in \bA_{sa}$, we have
\begin{equation}
\dist(U(a^{(n)}),U(b^{(n)})) = \dist(U(a),U(b))
\end{equation}
for all $n \in \bN$.

\begin{proof}
It is known that unital, simple, $\cZ$-stable C*-algebras are finite if and only if they have stable rank one \cite{Rordam04}, so the existence of a tracial state ensures that $\bA$ has stable rank one.
Moreover, $\cZ$-stable C*-algebras are known to have \emph{strict comparison of positive elements} \cite{Rordam04,MatuiSato12}, i.e., for $x,y \in \bA$, $d_\tau(x) < d_\tau(y)$ implies that $x \precsim y$, where the subequivalence is given by Cuntz subequivalence. As $d_\tau(p) = \tau(p)$ for a projection $p \in \bA$ and, for two projections $p,q \in \bA$, we have that $p \preceq q \iff p \lesssim q$ \cite[Proposition 2.1]{Rordam92}, it follows that such C*-algebras have strict comparison of projections.
The result now follows from the above.
\end{proof}
\end{cor}


The above class includes monotracial simple AF algebras whose pairing between $K_0$ and traces is injective (in particular, UHF algebras), Bunce-Deddens algebras \cite[Section V.3]{DavidsonBook1}, and irrational rotation algebras \cite[Chapter VI]{DavidsonBook1}. The former have real rank zero since this property is preserved by inductive limits, and are $\cZ$-stable by \cite[Theorem 5(a)]{JiangSu99}.
The latter two fall into the class of \emph{approximately divisible C*-algebras}, which are known to have real rank zero and be $\cZ$-stable by \cite{BKR92} and \cite[Theorem 2.3]{TomsWinter08}, respectively. One can further come up with many examples arising from classification: the $K$-theory and traces, together with the pairing between them, forms a complete invariant for the class of unital, simple, separable, nuclear, $\cZ$-stable C*-algebras satisfying the UCT; see \cite{CGSTW23} for recent developments and a classification of morphisms in terms of a slightly enlarged invariant.
With this, one can take large classes of unital, simple, separable, nuclear, $\cZ$-stable C*-algebras, insist that these C*-algebras are monotracial (which implies finiteness, which is equivalent to stable rank one in this setting), insist that the pairing map
\begin{equation}
\tau_*: K_0(\bA) \to \bR
\end{equation}
have dense range, which is equivalent to having real rank zero \cite{Rordam04} for these C*-algebras, and that the pairing map also be injective.
Thus, for example, whenever $G \subseteq \bR$ is a dense subgroup, one can find a monotracial C*-algebra $\bA$ with $K_0(\bA) = G$, with the pairing map being the identity map, such that $\bA$ satisfies the hypotheses  above.


\subsection{Self-adjoint elements with connected spectrum} \label{subs:positive elements with 01 spectrum}

Under the $\cZ$-stability assumption, the results in \cite{JacelonStrungToms15} (see also \cite{Cheong15,JacelonStrungVignati21}) allow one to remove the real rank zero assumption in order to compare unitary orbits of certain elements. However, it comes at the cost of restricting ourselves to self-adjoint elements with connected spectrum.
The key here is that, in highly regular C*-algebras, the distance between unitary orbits of certain positive elements agrees with
\begin{equation}
\begin{split}
d_W(a,b) := \inf \{r > 0 \mid & (a - (t + r))_+\precsim (b-t)_+ \\
&\text{ and } (b - (t + r))_+ \precsim (a - t)_+ \text{ for all } t > 0\}.
\end{split}
\end{equation}
Furthermore, this also agrees, for nice elements with spectrum equal to the unit interval $[0,1]$ in such C*-algebras, with the L\'{e}vy-Prokhorov distance $d_P(a,b)$, which is the infimum over all $r > 0$ such that
\begin{equation}
\mu_{\tau,a}(U_r) \geq \mu_{\tau,b}(U) \text{ and } \mu_{\tau,b}(U_r) \geq \mu_{\tau,a}(U)
\end{equation}
for every open $U \subseteq (0,1]$ and every $\tau \in T(\bA)$, where $\mu_{\tau,a},\mu_{\tau,b}$ are the Borel measures induced by $\tau$ on $C^*(a)$ and $C^*(b)$ respectively and
\[
U_r = \{t \mid \dist(t,U) < r\}. \]


This quantity $d_P(a,b)$ is clearly invariant under matrix amplifications since it only depends on the Borel measures induced by $\tau$ on $C^*(a) \simeq C^*(a^{(n)})$ and $C^*(b) \simeq C^*(b^{(n)})$.


\begin{prop}
Let $\bA$ be a unital, simple, separable, stably finite, $\cZ$-stable, exact C*-algebra and let $a, b \in \bA_{sa}$ be self-adjoint elements with connected spectrum $[\alpha, \beta]$.   Then
\begin{equation}
\dist(U(a),U(b)) = \dist(U(a^{(n)}),U(b^{(n)}))
\end{equation}
for all $n \in \bN$.
\begin{proof}

Let us first consider the case where $a,b \in \bA_+$ are positive elements with $\sigma(a) = \sigma(b) = [0,1]$.

Making note that $d_P$ is stable under matrix amplifications, this follows from the result \cite[Theorem 5.2]{JacelonStrungToms15} that says such elements in such a C*-algebra satisfy
\begin{equation}
\begin{split}
\dist(U(a),U(b)) &= d_P(a,b) \\
&= d_W(a,b).
\end{split}
\end{equation}

If $\sigma(a) = \sigma(b) = [\alpha, \beta] \subseteq \bR$, then we note that $a_0 := \frac{1}{\beta- \alpha}(a- \alpha 1_{\bA})$ and $b_0 := \frac{1}{\beta- \alpha}(b- \alpha 1_{\bA})$ are positive elements with spectrum equal to $[0,1]$.   From above,
\[
\dist(U(a_0),U(b_0)) = \dist(U(a_0^{(n)}),U(b_0^{(n)})). \]

But a routine calculation shows that $\dist(U(a), U(b)) = (\beta - \alpha) \dist(U(a_0), U(b_0))$ and similarly $\dist(U(a^{(n)}),U(b^{(n)})) = (\beta - \alpha) \dist(U(a_0^{(n)}),U(b_0^{(n)}))$, from which the result immediately follows.
\end{proof}
\end{prop}


\subsection{Normal elements in UHF algebras}\label{subs: normal elements in uhfs}
As UHF algebras are monotracial simple C*-algebras with real rank zero which fall into the scope of classification, Kaplansky's problem is known to hold for self-adjoint elements in UHF algebras by the results in Proposition \ref{prop: KPn for self-adjoints in nice C*-algebras}.

If $\bA$ is a unital C*-algebra admitting a tracial state $\tau$, we say that two elements $a, b \in \bA$ \emph{satisfy Specht's condition} (with respect to $\tau$) if
\[
\tau(w(a, a^*)) = \tau(w(b, b^*)) \]
for all words $w(x,y)$ in two non-commuting variables $x$ and $y$.   Specht~\cite{Specht1940} showed that two matrices $a, b \in \bM_n$ satisfy this condition if and only if $a$ is unitarily equivalent to $b$.   In passing from $\bM_n$ to (infinite-dimensional) C*-algebras, the continuity of $\tau$ implies that the most one might be able to hope for is that $a$ and $b$ satisfy Specht's condition if and only if $a$ is \emph{approximately} unitarily equivalent to $b$.   For normal elements of UHF-algebras, it was shown in \cite{MarcouxZhang21} that this is indeed the case.  (In the same paper, an example of two elements $a, b$ of the UHF-algebra $\bM_{2^\infty}$ which satisfy Specht's condition but which fail to be approximately unitarily equivalent was given, although they do become approximately unitarily equivalent when considered as elements of the universal UHF-algebra $\cQ$.)

\begin{prop}
Let $n \in \bN$, $\bA$ be a UHF algebra and $a,b \in \bA$ be normal. Then
\begin{equation}
a^{(n)} \simeq_a b^{(n)} \text{ if and only if } a \simeq_a b.
\end{equation}
\begin{proof}
It is clear that  $a^{(n)}$ and $b^{(n)}$ satisfy Specht's condition if and only if $a$ and $b$ do, since the trace is invariant under amplification (once we normalize the amplified trace). Thus, by \cite[Proposition 2.7]{MarcouxZhang21}, $a$ and $b$ are approximately unitarily equivalent.  The converse is straightforward.
\end{proof}
\end{prop}


In the next section (see Corollary \ref{cor: CAR elements not satisfying 3-KP}), we shall demonstrate the surprising result that even in UHF-algebras, it is possible to find (necessarily non-normal) elements $a$ and $b$ for which the distance between the unitary orbits of  \emph{some} amplifications decreases, while the distance between the unitary orbits of other amplifications of the same elements remains the same.


\subsection{Normal elements in von Neumann algebras}\label{subs:normal elements in vN}
As previously mentioned, it is known that  if $\cM$ is a general von Neumann algebra $\cM$ and if $a,b \in \cM$ are arbitrary, then for any $n \in \bN$, the  unitary equivalence of $a^{(n)}$ and $b^{(n)}$ in $\bM_n(\cM)$ implies the unitary equivalence of $a$ and $b$ in $\cM$, cf. \cite{Azoff95} and \cite[Appendix A.2]{Sherman10}. Here we give a partial treatment of the same question, except with unitary equivalence replaced with approximate unitary equivalence. We note that this version of Kaplansky's problem holds for arbitrary elements of $\bh$ by \cite[Proposition 4.4]{MRTZ24}.

In \cite{Sherman07Unitary}, closures of unitary orbits of normal operators in von Neumann algebras were considered.
In a von Neumann algebra $\cM$, a normal element $a \in \cM$, and an open set $U \subseteq \bC$, we will use $\chi_U(a)$ to denote the associated spectral projection of $a$.  We denote the set of all projections in $\cM$ by $P(\cM)$.   Murray-von Neumann equivalence of projections in $\cM$ is an equivalence relation, and we denote the equivalence classes of $\cM$ relative to this relation by $P(\cM)/\MvN$. The following is \cite[Theorem 1.3(2)]{Sherman07Unitary}.


\begin{theorem} \textbf{\emph{[Sherman]}}
Let $a,b \in \cM$ be two normal elements in a von Neumann algebra $\cM$. Then $a$ and $b$ are approximately unitarily equivalent if and only if for any open set $U \subseteq \bC$, $\chi_U(a) \MvN \chi_U(b)$.
\end{theorem}


Using the fact that the set of  projections modulo Murray-von Neumann equivalence in $\sigma$-finite factors form specific lattices, we can use Sherman's result above to show that pairs of normal elements in any $\sigma$-finite factor satisfy $[\text{KP}_n]$ for all $n \ge 1$.

We have the following identification, as lattices, based on the type (cf. \cite[Theorem III.1.7.9]{BlackadarBook} for example).
\begin{enumerate}
\item When $\cM$ is of type $\I_k$, where $k$ is finite or $\aleph_0$, $P(\cM)/\MvN \simeq \{0,1,\dots,k\}$.
\item When $\cM$ is of type $\II_1$, $P(\cM)/\MvN \simeq [0,1]$.
\item When $\cM$ is of type $\II_{\infty}$, $P(\cM)/\MvN \simeq [0,\infty]$.
\item When $\cM$ is of type $\III$, $P(\cM)/\MvN \simeq \{0,\infty\}$.
\end{enumerate}
We note that taking matrix amplifications of multiplicity $n \in \bN$ of a factor preserves its type, except in the $I_k$ case, where $k$ is finite. In this case, the $n$-th matrix amplification is of type $I_{nk}$.


\begin{lemma} \label{factorcase}
Let $\cM$ be a $\sigma$-finite factor, $n \in \bN$, and $a,b \in \cM$ be normal. Then
\begin{equation}
a^{(n)} \simeq_a b^{(n)} \text{ if and only if } a \simeq_a b.
\end{equation}
\begin{proof}
By Sherman's theorem, this amounts to showing that if $U \subseteq \bC$ is an open set and $\chi_U(a^{(n)}) \MvN \chi_U(b^{(n)})$, then $\chi_U(a) \MvN \chi_U(b)$.
However, this can easily be deduced based on type.
Each of the sets
\begin{equation}
\{0,1,\dots,k\},[0,\infty],\{0,\infty\}
\end{equation}
satisfy the property that if $i,j$ are in the set and $n \in \bN$, then
\begin{equation}
ni = nj \Rightarrow i = j.
\end{equation}
As the map $x \mapsto x^{(n)}$ induces the map $i \mapsto n\cdot i$, which is injective, the result follows for factors of type $\I_k,\II_{\infty},\III$.
For the $\II_1$ case, the map $x \mapsto x^{(n)}$ actually induces the identity map between $P(\cM)/\MvN$ and $P(\bM_n(\cM))/\MvN$, so that this case follows for free.
\end{proof}
\end{lemma}


The above allows us to adapt the result to general separably representable von Neumann algebra (which automatically implies $\sigma$-finite).
First we need to recall some facts about direct integrals -- we refer to the reader to sections 14.1 and 14.2 of \cite{KadisonRingroseII} and section IV.8 of \cite{TakesakiII} for the details.

Let $(\Omega,\mu)$ be a measure space. Suppose that
\begin{equation}
\cH = \int_\Omega^\oplus \cH_\omega d\mu(\omega)
\end{equation}
is the direct integral of Hilbert spaces over $(\Omega,\mu)$. We say an operator $T \in \bh$ is \emph{decomposable} if there is a function $T \mapsto T_\omega$ on $\Omega$ such that $T_\omega \in \cB(\cH_\omega)$, and for each $\xi = (\xi_\omega)$ in $\cH$, $T_\omega(\xi_\omega) = (T\xi)_\omega$ for almost every $\omega$. Further, we say $T$ is \emph{diagonal} if the assignment $T_\omega$ is just a multiple of the identity $I_\omega \in \cB(\cH_\omega)$ for almost every $\omega$. We note that the set of decomposable operators is precisely the set of operators which commute with the diagonal operators. The following is \cite[Theorem IV.8.21]{TakesakiII}.


\begin{theorem} \emph{\textbf{[Disintegration]}}

Every separably represented von Neumann algebra $\cM \subseteq \bh$ can be be written as the direct integral
\begin{equation}
\cM = \int_\Omega^\oplus \cM_\omega d\mu(\omega)
\end{equation}
over some measure space $(\Omega,\mu)$, where the centre of $\cM$ is precisely the von Neumann algebra of diagonal operators. In this case, every operator is decomposable and if $\cM$ is a von Neumann algebra of a certain type, say $X$, then $\cM_\omega$ is a factor of type $X$ for almost every $\omega$.
\end{theorem}


With this in hand, we are ready to prove that Kaplansky's problem has a positive answer if one restricts to normal elements in separably represented von Neumann algebras.

\begin{prop}
Let $\cM \subseteq \bh$ be a von Neumann algebra acting on a separable Hilbert space $\cH$, $n \in \bN$, and $a,b \in \cM$ be normal. Then
\begin{equation}
a^{(n)} \simeq_a b^{(n)} \text{ if and only if } a \simeq_a b.
\end{equation}
\begin{proof}

One direction is clear:  if $a \simeq_a b$, then obviously $a^{(n)} \simeq_a b^{(n)}$.

Suppose, therefore that $n \ge 2$ and that $a^{(n)} \simeq_a b^{(n)}$.

Every von Neumann algebra $\cM = \oplus_{j=1}^5 \cM_j$, where $\cM_1$ is of type $\I_k$, $\cM_2$ is of type $\I_\infty$, $\cM_3$ is of type $\II_1$, $\cM_4$ is of type $\II_\infty$ and $\cM_5$ is of type $\III$.   Thus we may write $a = \oplus_{j=1}^5 a_j$ and $b = \oplus_{j=1}^5 b_j$ relative to this decomposition.   As noted above, each of the above types are invariant under finite amplification, which implies that if $a^{(n)} \simeq_a b^{(n)}$, it follows that $a_j^{(n)} \simeq_a b_j^{(n)}$ for all $1 \le j \le 5$.     As such, if we can show that $\text{[KP}_n\text]$ holds for normal elements in von Neumann algebras of a fixed type, say $X$, then we may conclude that it holds in general.   Suppose, therefore, that $\cM$ is of type $X \in \{ \I_n, \I_\infty, \II_1, \II_\infty, \III \}$.

 Let us write
\begin{equation}
\cM = \int_\Omega^\oplus \cM_\omega d\mu(\omega)
\end{equation}
over a suitable measure space $(\cM,\mu)$ as above.

Now $\bM_n(\cM) = \int_{\Omega}^{\oplus} \bM_n(\cM_\omega) d\mu(\omega)$, and so if $a = \int_\Omega^\oplus a_\omega d\mu(\omega)$ and $b = \int_\Omega^\oplus b_\omega d\mu(\omega)$ relative to the above decomposition of $\cM$, then we may write
\[
a^{(n)} = \int_\Omega^\oplus a_\omega^{(n)} d\mu(\omega), \ \ \ \ \ b^{(n)} = \int_\Omega^\oplus b_\omega^{(n)} d\mu(\omega)\]
relative to the above decomposition of $\bM_n(\cM)$.

The hypothesis that $a^{(n)} \simeq_a b^{(n)}$ combined with Sherman's Theorem implies that
\begin{equation}
\chi_U(a_\omega)^{(n)} {\MvN} \chi_U(b_\omega)^{(n)}
\end{equation}
for almost all $\omega \in \Omega$.   But each $\cM_\omega$ is a factor, allowing us to conclude via Lemma~\ref{factorcase} that $\chi_U(a_\omega) {\MvN} \chi_U(b_\omega)$ a.e. with respect to $\omega \in \Omega$.   From this it follows that $\chi_U(a) \MvN \chi_U(b)$ a.e., and thus $a \simeq_a b$.
\end{proof}
\end{prop}


As the center-valued trace $\text{ctr}: \cM \to Z(\cM)$ on a type $\II_1$ von Neumann algebra (separably represented or not) determines the Murray-von Neumann equivalence class of a projection, it is in fact true that Kaplansky's problem holds for normal elements in a general type $\II_1$ von Neumann algebra.

\begin{prop}
Let $\cM$ be a type $\II_1$ von Neumann algebra, $n \in \bN$, and $a,b \in \cM$ be normal. Then
\begin{equation}
a^{(n)} \simeq_a b^{(n)} \text{ if and only if } a \simeq_a b.
\end{equation}
\begin{proof}
By Theorems 8.4.3(iii) and 8.4.4 of \cite{KadisonRingroseII}, the centre-valued trace induces an order isomorphism
\begin{equation}
P(\cM)/\MvN \simeq Z(\cM)_{+,1}
\end{equation}
between the equivalence classes of projections and the set $Z(\cM)_{+,1}$ of positive contractions in the centre of $\cM$.
Taking matrix amplifications just induces the identity map
\begin{equation}
Z(\cM)_{+,1} \simeq P(\cM)/\MvN \to P(\bM_n(\cM))/\MvN \simeq Z(\cM)_{+,1},
\end{equation}
where we are identifying the centres of $\cM$ and $\bM_n(\cM)$.
As this map is injective, we have as above that, whenever $a,b \in \cM$ are normal elements and $U \subseteq \bC$ is an open subset, $\chi_U(a^{(n)}) \MvN \chi_U(b^{(n)})$ if and only if $\chi_U(a) \MvN \chi_U(b)$. Sherman's Theorem then allows us to conclude that $a^{(n)}$ and $b^{(n)}$ are approximately unitarily equivalent if and only if $a$ and $b$ are.
\end{proof}
\end{prop}


\vspace {1 cm}

\section{Arbitrary elements in UHF-stable C*-algebras}\label{sec: arbitrary elements in uhf-stable}

\subsection{Kaplansky's problem for $\bM_{n^{\infty}}$-stable C*-algebras}


We now investigate the following interesting phenomenon:  when $\bA$ is a unital separable C*-algebra which is \emph{UHF-stable} -- that is, $\bA \otimes \bB \simeq \bA$ for some infinite-dimensional UHF algebra $\bB$ -- the distance between unitary orbits of \emph{any} two elements $a,b \in \bA$ is equal to the distance between the unitary orbits of $a^{(m)}$ and $b^{(m)}$ whenever $m^\infty$ divides the supernatural number of $\bB$.
To prove this, we will use techniques from the theory of strongly self-absorbing C*-algebras.


\begin{defn} \cite{TomsWinter07}
A unital separable C*-algebra $\cD$ is \textbf{strongly self-absorbing} if $\cD \not\simeq \bC$ and there is an isomorphism $\cD \to \cD \otimes \cD$ which is approximately unitarily equivalent to the first factor embedding $d \mapsto d \otimes 1_\cD$.
\end{defn}


All known strongly self-absorbing C*-algebras are: the Jiang-Su algebra $\cZ$ \cite{JiangSu99}, the Cuntz algebras $\cO_2$ and $\cO_{\infty}$ \cite{Cuntz77}, UHF algebras of infinite type (i.e., those whose supernatural numbers have an infinite power whenever a prime appears), and $\cO_{\infty}$ tensor a UHF algebra of infinite type. Moreover, these are in fact all of them if the UCT problem holds (the problem being: whether or not all unital, separable, simple, nuclear C*-algebras satisfy the UCT). Such C*-algebras have approximately inner tensor flip, are nuclear, simple and have at most one tracial state.

We note that as these C*-algebras have approximately inner flip, we can replace the condition that $d \mapsto d \otimes 1_\cD$ is approximately unitarily equivalent to an isomorphism with the condition that $d \mapsto 1_\cD \otimes d$ is.


\begin{lemma}
Let $\bA$ be a unital, separable C*-algebra. Suppose that $\varphi \in \End(\bA)$ is approximately unitarily equivalent to an automorphism $\psi \in \Aut(\bA)$. Then
\begin{equation}
\dist(U(a),U(b)) = \dist(U(\varphi(a)),U(\varphi(b)))
\end{equation}
for any $a,b \in \bA$.
\begin{proof}
As $\varphi$ and $\psi$ are approximately unitarily equivalent, we have that
\begin{equation}
\dist(U(\psi(a)),U(\varphi(a))) = \dist(U(\psi(b)),U(\varphi(b))) = 0.
\end{equation}
Consequently
\begin{equation}
\begin{split}
\dist(U(a),U(b)) &= \dist(U(\psi(a)),U(\psi(b)))\\
&= \dist(U(\varphi(a)),U(\varphi(b))).
\end{split}
\end{equation}
\end{proof}
\end{lemma}



\begin{prop}\label{prop: Mn-stable distance}
Let $n \in \bN$ and $\bA$ be a unital, separable C*-algebra such that $\bM_{n^{\infty}} \otimes \bA \simeq \bA$.
Then
\begin{equation}
\dist(U(a),U(b)) = \dist(U(a^{(n)}),U(b^{(n)}))
\end{equation}
for all $a,b \in \bA$.
\begin{proof}
Let us identify $\bA$ with $\bM_{n^{\infty}} \otimes \bA$ and note that up to unitary equivalence in $\bM_n(\bA)$, the map $a \mapsto a^{(n)}$ can be identified with the map $\varphi: \bM_{n^{\infty}} \otimes \bA  \to \bM_n \otimes \bM_{n^{\infty}} \otimes \bA$ given by
\begin{equation}\label{eq:first factor embedding}
\varphi(x) = 1_{\bM_n} \otimes x = x^{(n)}.
\end{equation}
This can be identified with the map $\varphi = \varphi_0 \otimes \id_\bA$ where $\varphi_0: \bM_{n^{\infty}} \to \bM_n \otimes \bM_{n^{\infty}}$ is given by the same formula (\ref{eq:first factor embedding}).
The C*-algebra $\bM_{n^{\infty}}$ is clearly isomorphic to $\bM_n \otimes \bM_{n^{\infty}}$, and, as this C*-algebra is strongly self-absorbing, \cite[Corollary 1.12]{TomsWinter07} gives that $\varphi_0$ is approximately unitarily equivalent to an isomorphism $\psi: \bM_{n^{\infty}} \simeq \bM_n \otimes \bM_{n^{\infty}}$.
Thus $\varphi = \varphi_0 \otimes \id_\bA$ is approximately unitarily equivalent to an automorphism $\psi \otimes \id_\bA$.
The result follows from the above lemma.
\end{proof}
\end{prop}


As a consequence, we recover a positive solution to Kaplansky's problem for $m$th matrix amplifications in the setting of unital, separable, $\bM_{n^{\infty}}$-stable C*-algebras, whenever $m$ divides $n$.

\begin{cor} \label{cor:4.4}
Let $n \in \bN$ and suppose that $\bA$ is a unital separable C*-algebra such that $\bA \simeq \bM_{n^{\infty}} \otimes \bA$. Then, for $a,b \in \bA$ and $m \in \bN$ such that $m | n$, we have
\begin{equation}
a^{(m)} \simeq_a b^{(m)} \iff a \simeq_a b.
\end{equation}
\begin{proof}
If $m | n$, then $\bM_{m^{\infty}} \otimes \bM_{n^{\infty}} \simeq \bM_{n^{\infty}}$, and thus $\bA$ absorbs $\bM_{m^{\infty}}$ if $\bA$ absorbs $\bM_{n^{\infty}}$. Now the result follows from the above.
\end{proof}
\end{cor}


Recall that $\cQ$ denotes the \emph{universal} UHF C*-algebra, that is: the UHF C*-algebra whose supernatural number is divisible by any integer.   From above we deduce that  $\cQ$-stable C*-algebras will have a positive solution for any $n \in \bN$. This can be viewed as an analogue of a positive answer to Kaplansky's question in the case of infinite abelian groups where the groups $G$ and $H$ are divisible (see \cite[Exercise 9]{Kaplansky54}).

\begin{cor}\label{cor:Q-stable algebras satisfy KP}
Let $\bA$ be a unital, separable, $\cQ$-stable C*-algebra. Then, for $a,b \in \bA$, the following are equivalent.
\begin{enumerate}
\item $a \simeq_a b$ in $\bA$;
\item $a^{(n)} \simeq_a b^{(n)}$ in $\bM_n(\bA)$ for any $n \geq 1$;
\item $a^{(n)} \simeq_a  b^{(n)}$ in $\bM_n(\bA)$ for some $n \geq 1$.
\end{enumerate}
\end{cor}


In particular, the above holds for the Cuntz algebra $\cO_2$ \cite{Cuntz77}.

\begin{remark}
In fact, by \cite[Theorem 2.2]{DadarlatWinter09}, any two embeddings of a strongly self-absorbing C*-algebra $\cD$ into $\cD \otimes \bA$ are (strongly) asymptotically unitarily equivalent. Thus one can argue in the same way to conclude that if $\bA \simeq \bM_{n^{\infty}} \otimes \bA$ is unital and separable, $a,b \in \bA$, and $a^{(n)}$ and $b^{(n)}$ are asymptotically unitarily equivalent in $\bM_n(\bA)$, then so are $a$ and $b$ in $\bA$.
\end{remark}


\begin{remark}
$\II_1$ factors are additionally equipped with the $\|\cdot\|_2$-norm, explicitly given by $\|a\|_2 = \tau(a^*a)^\frac{1}{2}$, where $\tau$ is the unique faithful normal tracial state. In this setting, when $\cM$ is separable (with respect to the $\|\cdot\|_2$-norm) and McDuff -- that is, $\cM \simeq \cM {\ov} \cR$, where $\cR$ is the hyperfinite $\II_1$ factor -- the strongly self-absorbing results have analogues and we will indeed get that
\begin{equation}
\dist_{\|\cdot\|_2}(U(a^{(n)}),U(b^{(n)})) = \dist_{\|\cdot\|_2}(U(a),U(b))
\end{equation}
for arbitrary elements $a,b \in \cM$.
This is because any two embeddings of $\cR$ into $\cM \ov \cR \simeq \cM$ will be approximately unitarily equivalent in the point-$\|\cdot\|_2$-norm topology, and then one can argue as in Proposition \ref{prop: Mn-stable distance}.\footnote{In fact, it is true that there is a unique embedding of $\cR$ into any separable $\II_1$ factor, up to $\|\cdot\|_2$-approximate unitary equivalence. However to replicate the argument in Proposition \ref{prop: Mn-stable distance}, we still require $\cR \ov \cM \simeq \cM$.}
\end{remark}


\subsection{A counter-example in the CAR algebra}
For the CAR algebra $\bM_{2^\infty}$, it follows from Corollary~\ref{cor:4.4} that $a^{(2)} \simeq_a b^{(2)}$ implies that $a \simeq_a b$, and more generally, for any $n \ge 1$, $a^{(2^n)} \simeq_a b^{(2^n)}$ implies that $a \simeq_a  b$.   It is therefore natural to ask whether or not $a^{(3)} \simeq_a b^{(3)}$ implies that $a \simeq_a b$.

Recall that, for $n \in \bN$, the (\emph{unital}) \emph{dimension drop algebra} $\bI_n$ is the C*-algebra
\begin{equation}
\bI_n := \{f \in C([0,1],\bM_n) \mid f(0),f(1) \in \bC\cdot I_n\}
\end{equation}
of functions from $[0,1]$ into $\bM_n$ which take scalar values at the endpoints.
The $K$-theory of $\bI_n$ is as follows: $K_0(\bI_n) = \bZ$ and $K_1(\bI_n) = \bZ_n$, the cyclic group of order $n$.
This is an interesting C*-algebra for various reasons relating to $K$-theory, and in particular mod-$p$ $K$-theory. One feature of this algebra is the fact that evaluation at the different endpoints correspond to different $KK$-classes, see \cite{DadarlatLoring96Modp}.


\begin{theorem} \label{thm:thm4.8}
There are two embeddings $\Phi,\Psi: \bI_3 \to \bM_{2^{\infty}}$ such that
\begin{enumerate}
\item $\Phi^{(3)} \simeq_a \Psi^{(3)}$ in $\bM_3(\bM_{2^{\infty}})$;
\item $\Phi \not\simeq_a \Psi$ in $\bM_{2^{\infty}}$.
\end{enumerate}
\end{theorem}

\begin{proof}
We use the embeddings $\Phi,\Psi:\bI_3 \into \bM_{2^{\infty}}$ constructed in \cite[Theorem 3.5]{MarcouxZhang21}, where it is shown that (2) holds yet $\tau(\Phi(f)) = \tau(\Psi(f))$ for all $f \in \bI_3$. We note that this trace condition follows from (1), as the trace is invariant under approximate unitary equivalence. Borrowing notation from the proof of that result, let us fix $(f_n) \subseteq \bI_3$ to be a dense sequence of the unit ball of $\bI_3$, and let $\ee_n = \frac{1}{2^n}$. As each $f_i$ is uniformly continuous, find $2 \leq d_n \in \bN$ such that for $1 \le i \le n$,
\begin{equation}
\|f_i(t) - f_i(s)\| < \ee_n
\end{equation}
whenever $|s - t| \leq \frac{1}{d_n}$. Without loss of generality, let us assume that the sequence $(d_n)$ is strictly increasing. Set
\begin{equation}
l_n = \frac{4^{d_n} - 1}{3},\ \ \  r_n = \frac{1}{l_n + 1} = \frac{3}{4^{d_n} + 2},\ \ \  m_n = 4^{d_{n+1} - d_n}
\end{equation}
and
\begin{equation}
t_j^{(n)} = j\cdot r_n, j=1,\dots,l_n.
\end{equation}
Define $\underline{f}(t)$ to be the (1,1) corner of $f(t)$, so that $\underline{f}(0)$ is just the scalar $\lambda \in \bC$ such that $f(0) = \lambda \cdot I_3$ and $\underline{f}(1)$ is just the scalar $\mu \in \bC$ such that $f(1) = \mu \cdot I_3$.
\begin{equation}
\varphi_n(f) := \underline{f}(0) \oplus f(t_1^{(n)}) \oplus \cdots \oplus f(t_{l_n}^{(n)})
\end{equation}
and
\begin{equation}
\psi_n(f) := f(t_1^{(n)}) \oplus \cdots \oplus f(t_{l_n}^{(n)}) \oplus \underline{f}(1).
\end{equation}
Letting $\alpha_n,\beta_n: \bM_{d_n} \to \bM_{d_{n+1}}$ be embeddings given by the diagonal embedding composed with conjugation by a unitary which takes elements of the form
\begin{equation}
I_{m_n} \otimes \begin{pmatrix}
\lambda \\ & a_1 \\ & & \ddots \\ & & & a_{l_n}
\end{pmatrix} \text{ and }
I_{m_n} \otimes \begin{pmatrix}
b_1 \\ &  \ddots \\ & &  b_{l_n} \\ & & & \mu
\end{pmatrix}
\end{equation}
to
\begin{equation}
\begin{pmatrix}
I_{m_n} \otimes \lambda \\ & I_{m_n} \otimes a_1\\ & & \ddots \\ & & & I_{m_n} \otimes a_{l_n}
\end{pmatrix}
\end{equation}
and
\begin{equation}
\begin{pmatrix}
I_{m_n} \otimes b_1 \\ &  \ddots \\ & & I_{m_n} \otimes  b_{l_n} \\ & & & I_{m_n} \otimes \mu
\end{pmatrix}
\end{equation}
respectively. This yields approximately commuting diagrams
\begin{equation}
\begin{CD}
\mathbb{I}_3 @>{\id}>> \mathbb{I}_3  @>{\id}>> \mathbb{I}_3 @>{\id}>> \cdots @>{}>>\mathbb{I}_3\\
@VV{\varphi_1}V @VV{\varphi_2}V @VV{\varphi_3}V@. @VV{\Phi}V\\
\bM_{4^{d_1}}@>{\alpha_1}>> \bM_{4^{d_2}} @>{\alpha_2}>> \bM_{{4^{d_3}}} @>{\alpha_3}>> \cdots @>{}>>\bM_{2^\infty}
\end{CD}
\end{equation}
and
\begin{equation}
\begin{CD}
\mathbb{I}_3 @>{\id}>> \mathbb{I}_3  @>{\id}>> \mathbb{I}_3 @>{\id}>> \cdots @>{}>>\mathbb{I}_3\\
@VV{\psi_1}V @VV{\psi_2}V @VV{\psi_3}V@. @VV{\Psi}V\\
\bM_{4^{d_1}}@>{\alpha_1}>> \bM_{4^{d_2}} @>{\alpha_2}>> \bM_{{4^{d_3}}} @>{\alpha_3}>> \cdots @>{}>>\bM_{2^\infty}
\end{CD}
\end{equation}
where $\Phi$ and $\Psi$ exist since $\sum_n \ee_n = 1$ by \cite[Theorem 1.10.14]{LinBook}.
It is also clear that $\Phi^{(3)}$ and $\Psi^{(3)}$ can be realized as the limits of $\varphi_n^{(3)}$ and $\psi_n^{(3)}$ respectively, as we further have the approximately commuting diagrams
\begin{equation}
\begin{CD}
\mathbb{I}_3 @>{\id}>> \mathbb{I}_3  @>{\id}>> \mathbb{I}_3 @>{\id}>> \cdots @>{}>>\mathbb{I}_3\\
@VV{1_3\otimes \varphi_1}V @VV{1_3\otimes\varphi_2}V @VV{1_3\otimes\varphi_3}V@. @VV{\Phi^{(3)}}V\\
\bM_3\otimes \bM_{4^{d_1}}@>{\id_{\bM_3}\otimes \alpha_1}>> \bM_3\otimes\bM_{4^{d_2}} @>{\id_{\bM_3}\otimes\alpha_2}>> \bM_3\otimes\bM_{{4^{d_3}}} @>{\id_{\bM_3}\otimes\alpha_3}>> \cdots @>{}>>\bM_3\otimes\bM_{2^\infty}
\end{CD}
\end{equation}
and
\begin{equation}
\begin{CD}
\mathbb{I}_3 @>{\id}>> \mathbb{I}_3  @>{\id}>> \mathbb{I}_3 @>{\id}>> \cdots @>{}>>\mathbb{I}_3\\
@VV{1_3\otimes \psi_1}V @VV{1_3\otimes\psi_2}V @VV{1_3\otimes\psi_3}V@. @VV{\Psi^{(3)}}V\\
\bM_3\otimes \bM_{4^{d_1}}@>{\id_{\bM_3}\otimes \beta_1}>> \bM_3\otimes\bM_{4^{d_2}} @>{\id_{\bM_3}\otimes\beta_2}>> \bM_3\otimes\bM_{{4^{d_3}}} @>{\id_{\bM_3}\otimes\beta_3}>> \cdots @>{}>>\bM_3\otimes\bM_{2^\infty}
\end{CD}
\end{equation}
However, $\varphi_n^{(3)}(f)$ is unitarily equivalent to
\begin{equation}
\varphi_n'(f) = f(0) \oplus f(t_1^{(n)})^{(3)} \oplus \cdots \oplus f(t_{l_n}^{(n)})^{(3)}
\end{equation}
and $\psi_n^{(3)}$ is unitarily equivalent to
\begin{equation}
\psi_n'(f) = f(t_1^{(n)})^{(3)} \oplus \cdots \oplus f(t_{l_n}^{(n)})^{(3)} \oplus f(1).
\end{equation}
Moreover, the unitaries implementing these equivalences do not depend on $f$. By the continuity of $f$, together with the way we chose our points $t_i^{(n)}$, we have that $\|\varphi_n'(f) - \psi_n'(f)\| \to 0$ as $n \to \infty$ because this norm is just
\begin{equation}
\max\{\|f(0) - f(t_1^{(n)})\|, \|f(t_1^{(n)}) - f(t_{2}^{(n)})\|, \dots, \|f(t_{l_{n-1}}^{(n)}) - f(t_{l_n}^{(n)})\|, \|f(t_{l_n}^{(n)}) - f(1)\|\},
\end{equation}
which tends to 0 as $n \to \infty$.
As such, we have that
\begin{equation}
\dist(U(\varphi_n^{(3)}(f)),U(\psi_n^{(3)}(f))) \to 0
\end{equation}
as $n \to \infty$, which in turn yields that $\Phi^{(3)}$ is approximately unitarily equivalent to $\Psi^{(3)}$ in $\bM_3 \otimes \bM_{2^{\infty}} = \bM_3(\bM_{2^{\infty}})$.
\end{proof}


Since these elements were already known to satisfy Specht's condition from \cite{MarcouxZhang21}, this begs the question: do two elements in an arbitrary UHF algebra satisfy Specht's condition if and only if there is a suitable number $n$ for which their $n$-amplifications are approximately unitarily equivalent?  We have not yet been able to resolve this question.

Nonetheless, in the UHF setting, Question \ref{factor-pair-amplification} has a negative answer.
Indeed, it appears as though there need be no relationship in general between the distances between the unitary orbits of amplifications of a given pair $a, b \in \bA$ of different multiplicities.


\begin{cor}\label{cor: CAR elements not satisfying 3-KP}
There are elements $a,b \in \bM_{2^{\infty}}$ and $\delta > 0$ such that
\begin{enumerate}
\item $\dist(U(a^{(2^k)}),U(b^{(2^k)})) = \delta$ for all $k \in \bN$;
\item $\dist(U(a^{(3k)}),U(b^{(3k)})) = 0$ for all $k \in \bN$.
\end{enumerate}
\begin{proof}
By \cite[Proposition 3.7]{MarcouxZhang21}, $\bM_2(\bI_3) = \bM_2 \otimes \bI_3$ is singly generated, say with generator $x$. Then, letting $\Phi,\Psi$ be as in the above theorem, we can set
\begin{equation}
a = \id_{\bM_2} \otimes \Phi(x) \in \bM_2 \otimes \bM_{2^{\infty}}
\end{equation}
and
\begin{equation}
b = \id_{\bM_2} \otimes \Psi(x) \in \bM_2 \otimes \bM_{2^{\infty}}.
\end{equation}
Let us identify $\bM_2 \otimes \bM_{2^{\infty}}$ with $\bM_{2^{\infty}}$ in order to think of $a$ and $b$ as elements of $\bM_{2^{\infty}}$.
As $\Phi^{(3)}$ and $\Psi^{(3)}$ are approximately unitarily equivalent, we have that
\begin{equation}
\dist(U(a^{(3k)}),U(b^{(3k)})) \leq \dist(U(a^{(3)}),U(b^{(3)})) = 0,
\end{equation}
and hence $a^{(3k)} \simeq_a b^{(3k)}$ for all $k \in \bN$. Further, as $a,b$ are not approximately unitarily equivalent, set
\begin{equation}
\delta = \dist(U(a),U(b)) > 0.
\end{equation}
Now, as $\bM_{2^{\infty}}$ is $\bM_{2^{\infty}}$-stable, it follows from Proposition \ref{prop: Mn-stable distance} that
\begin{equation}
\delta = \dist(U(a^{(2^k)}),U(b^{(2^k)}))
\end{equation}
for all $k \in \bN$.
\end{proof}
\end{cor}


As we have just seen, it is possible to find $a, b \in \bM_{2^\infty}$ such that $a^{(3)} \simeq_a b^{(3)}$ although $a \not \simeq_a b$, and the obstruction lies at the level of $a^{(2)}$ and $b^{(2)}$.   As we now demonstrate, if we now \emph{suppose} that \emph{both} $a^{(3)} \simeq_a b^{(3)}$ and $a^{(2)} \simeq_a b^{(2)}$, then we may indeed conclude that $a \simeq_a b$.

In fact, we prove something more general.  Before doing so, we require a small technical lemma.

\begin{lemma}
Let $2 \le n \in \bN$ and suppose that $p$ and $q$ are coprime positive integers.  Then there exist positive integers $c_1$, $c_2$, and $r$ such that
\[c_1p+c_2q=n^r.\]

\begin{proof}
By Bezout's theorem, there exists integers $l_1, l_2$ such that $l_1p+l_2q=1$. It is easy to see that one of $l_1, l_2$
is positive, and the other one must be negative. Without loss of generality, we assume that $l_1<0 < l_2$. Since  $n\geq 2$, we can
choose a sufficiently large positive integer $r$ for which
\[n^r\geq 2p|l_1 p|+p.\]

Next, choose $k,d\in \mathbb{N}$,
such that
\[n^r=kp+d,\]
where $1\leq d\leq p$.
Since
\[
n^r\geq 2p|l_1 p|+p, \]
it follows that $k\geq 2|l_1p|$. Therefore,
\[n^r=kp+d\cdot 1=kp+d(l_1p+l_2q)=(k+dl_1)p+dl_2q.\]
Set $c_1=k+dl_1, c_2=dl_2$, and note that
\[c_1=k+dl_1\geq k+pl_1\geq |l_1 p|,\]
so that $c_1$ is a positive integer, as is $c_2$. Clearly
\[c_1p+c_2q=n^r,\]
completing the proof.
\end{proof}
\end{lemma}

As a consequence of this Lemma and Corollary~\ref{cor:4.4}, we obtain the following result.

\begin{theorem}\label{main theorem 1}
Let $p$ and $q$ are coprime positive integers, $n \geq 2$ is an integer, $\bA$ be a unital separable C*-algebra satisfying
$\bA \simeq \bM_{n^{\infty}} \otimes \bA$. Suppose that $a,b\in \bA$ satisfy $a^{(p)}\simeq_a b^{(p)}$ in $\bM_{p}(\bA)$, and
$a^{(q)}\simeq_a b^{(q)}$ in $\bM_{q}(\bA)$. Then $a$ and $b$ are approximately unitarily equivalent in $\bA$.
\end{theorem}

\begin{proof}
Using the previous Lemma, we may choose positive integers $c_1, c_2$ and $r \ge 2$ such that $n^r = c_1 p +c_2 q$.    The hypotheses then imply that
\[
a^{(n^r)} = (a^{ (p)})^{(c_1)} \oplus (a^{(q)})^{(c_2)} \simeq_a (b^{ (p)})^{(c_1)} \oplus (b^{(q)})^{(c_2)}  \simeq b^{(n^r)}. \]
The result is now follows from Corollary~\ref{cor:4.4}.
\end{proof}


The analysis of Theorem~\ref{thm:thm4.8} and Corollary~\ref{cor: CAR elements not satisfying 3-KP} can be used to show that there exist $n \in \bN$ and $a,b \in \bM_{2^{n}}$ such that
\begin{equation}
\dist(U(a),U(b)) > \dist(U(a^{(3)}),U(b^{(3)})).
\end{equation}
In fact, the proof actually gives us the following, which allows us to conclude that the distances between unitary orbits in Question \ref{factor-pair-amplification} will have no relationship, even for matrices.


\begin{prop}\label{prop: existence of non-KP matrices}
Fix $k \in \bN \cup \{0\}$ and $l \in \bN$. There exist a positive integer $n \in \bN$ and matrices $a,b \in \bM_{2^{n}}$ such that
\begin{equation}
\dist(U(a^{(2^{k})}),U(b^{(2^{k})}))  > \dist(U(a^{(3l)}),U(b^{(3l)})).
\end{equation}
\begin{proof}
We shall argue by contradiction. Suppose not; that is, suppose that
\begin{equation}
\dist(U(a^{(2^k)}),U(b^{(2^k)})) \leq \dist(U(a^{(3l)}),U(b^{(3l)}))
\end{equation}
for all integers $n$ and pairs of matrices $a,b \in \bM_{2^{n}}$.

Let $a,b \in \bM_{2^{\infty}}$ be the elements appearing in Corollary~\ref{cor: CAR elements not satisfying 3-KP}, and without loss of generality assume that $\|a\| = \|b\| = 1$. Note that $a^{(3)} \simeq_a b^{(3)}$ implies that $a^{(3l)} \simeq_a b^{(3l)}$.
Let $0 < \ee < 1$ and choose a unitary element $u \in \bM_{3l}(\bM_{2^{\infty}})$ such that $\|b^{(3l)} - u^*a^{(3l)}u\| < \ee$.

Now $a,b \in \bM_{2^{\infty}}$ implies that there is an integer $n \in \bN$, $a_n,b_n \in \bM_{2^n} \subseteq \bM_{2^{\infty}}$ and $u_n \in U(\bM_{3l}(\bM_{2^{n}})) \subseteq \bM_{3l}(\bM_{2^{\infty}})$ such that $\|a_n - a\|, \|b_n - b\|, \|u_n - u\| < \ee$ (the fact that we can approximate unitaries in the inductive limit C*-algebra by unitaries from its building blocks can be found as \cite[Lemma 2.7.2(4)]{LinBook}).
Applying the triangle inequality multiple times shows that
\begin{equation}
\|b_n^{(3l)} - u_n^*a_n^{(3l)}u_n\| < 5\ee.
\end{equation}
It then follows that
\begin{equation}
\dist(U(a_n^{(2^k)}),U(b_n^{(2^k)})) < 5\ee.
\end{equation}
That is, we can now find a unitary $v_n \in U(\bM_{2^k}(\bM_{2^n})) \subseteq \bM_{2^{\infty}}$ such that
\begin{equation}
\|b_n^{(2^k)} - v_n^*a_n^{(2^k)}v_n\| < 5\ee.
\end{equation}
Once again, the triangle inequality yields that
\begin{equation}
\|b^{(2^k)} - v_n^*a^{(2^k)}v_n\| < 7\ee.
\end{equation}
Since $\ee$ was arbitrary, this implies that $a$ and $b$ are approximately unitarily equivalent in $\bM_{2^{\infty}}$. This is a contradiction.
\end{proof}
\end{prop}

One is led to naturally ask: how small can we choose the power of 2?


\begin{question}
Are there matrices $S,T \in \bM_2$ such that
\begin{equation}
\dist(U(S),U(T)) > \dist(U(S^{(3)}),U(T^{(3)}))?
\end{equation}
\end{question}


\begin{question}
Suppose that $n \in \bN$ and $k$ is co-prime to $n$. Can one find matrices $S,T \in \bM_k$ such that
\begin{equation}
\dist(U(S),U(T)) > \dist(U(S^{(n)}),U(T^{(n)}))?
\end{equation}
What can be said if we do not assume that $n$ and $k$ are coprime? Say $d = \gcd(n,k)$. Do we necessarily have that
\begin{equation}
\dist(U(S),U(T)) = \dist(U(S^{(d)}),U(T^{(d)}))
\end{equation}
whenever $S,T \in \bM_k$?
\end{question}


\section{Purely infinite C*-algebras}


\subsection{Normal elements in unital simple purely infinite C*-algebras}\label{subs: normal elements in pi}

It is well-known that the unitary orbits of normal elements in the Calkin algebra $\cQ(\cH) = \cB(\cH)/\bK$ are classified by their spectrum and their index. Indeed, the Brown-Douglas-Fillmore Theorem \cite{BrownDouglasFillmore06} states that two essentially normal operators $N$ and $M$ are  unitarily equivalent if and only if they have the same spectrum and their semi-Fredholm index functions agree (see \cite[Section II.4]{DavidsonBook1}).

Recall that for an operator $T \in \bh$, $\sigma_e(T)$ denotes the \emph{essential spectrum} of $T$, which is the spectrum of $\pi(T)$ in $\bh/\bK$, where $\pi: \bh \to \bh/\bK$ is the quotient map.
If $T$ is invertible modulo the compacts, then both $T$ and $T^*$ have finite-dimensional kernels and the index
\begin{equation}
\ind(T) = \dim\ker T - \dim\ker T^* \in \bZ
\end{equation}
is well-defined.
We note that the index is additive with respect to direct sums: namely, if the index of $S$ and $T$ are defined, then
\begin{equation}
\ind (S \oplus T) = \ind (S) \oplus \ind (T).
\end{equation}
In particular, $\ind(T^{\oplus n}) = n\cdot \ind (T)$.


\begin{theorem} \textbf{\emph{[BDF]}}
Let $M,N \in \bh$ be essentially normal operators. Then $\pi(M)$ and $\pi(N)$ are  unitarily equivalent if and only if
\begin{enumerate}
\item $\sigma_e(M) = \sigma_e(N)$ and
\item $\text{ind}(\lambda\cdot 1 - M) = \text{ind}(\lambda\cdot 1 - N)$ for all $\lambda \notin \sigma_e(M) = \sigma_e(N)$.
\end{enumerate}
\end{theorem}

It follows from the continuity of the spectrum of normal elements of the Calkin algebra (with respect to the Hausdorff distance on compact subsets of $\bC$) and the continuity of the Fredholm index that unitary orbits of normal elements of the Calkin algebra are closed.   As such, two normal elements in the Calkin algebra are approximately unitarily equivalent if and only if they are unitarily equivalent, which in turn coincides with the condition that they have the same spectrum and same index function.  Thus each instance of unitary equivalence in the next Corollary can be replaced with instances of approximate unitary equivalence.  We point out that not all unitary orbits in the Calkin algebra are closed.  For example, Davidson~\cite{Davidson84} has shown that if $W$ is the Kakutani shift, then the unitary orbit of $\pi(W)$ is not closed in the Calkin algebra.


\begin{cor}
Let $a,b \in \bA := \bh/\bK$ be two normal elements in the Calkin algebra. The following are equivalent.
\begin{enumerate}
\item $a \simeq b$ in $\bA$;
\item $a^{(n)} \simeq b^{(n)}$ in $\bM_n(\bA)$ for all $n \in \bN$;
\item $a^{(n)} \simeq b^{(n)}$ in $\bM_n(\bA)$ for some $n \in \bN$.
\end{enumerate}
\begin{proof}
We only need to show (3) implies (1). Note that the spectrum is invariant under taking matrix amplifications, and that the index is additive with respect to direct sums. As the index is integer valued, coupled with the fact that
\begin{equation}
nx = ny \Rightarrow x = y \text{ whenever } x,y \in \bZ \text{ and } n \in \bN,
\end{equation}
it follows that $a \simeq_a b$ in $\bA$ by BDF.
\end{proof}
\end{cor}


The Fredholm index of an essentially invertible operator is precisely the $K_1$-class of the image of the operator under the quotient map (or equivalently the coset of the connected component of the identity in the set of invertible elements of $\bA$ to which it belongs), and the index is precisely the index map $\partial_1: K_0(\bK) \to K_1(\bh/\bK)$ arising in the 6-term exact sequence for topological $K$-theory.
A generalization of the above is the following.

\begin{theorem}\cite[Theorem 1.7]{Dadarlat95approx}
Let $X$ be a compact metrizable space. Let $\bB$ be a unital, simple, purely infinite C*-algebra and let $\varphi,\psi: C(X) \to \bB$ be two unital *-homomorphisms. Then $[\varphi] = [\psi]$ in $KL(C(X),\bB)$ if and only if $\varphi$ is approximately unitarily equivalent to $\psi$.
\end{theorem}

Here, $KL$ is a quotient of $KK$, which is invariant under approximate unitary equivalence. As $KK$ (or $KL)$ recovers both the $K_0$ and $K_1$ groups (indeed, $C(X)$ satisfies the UCT, so there will always be a relationship between $KK(C(X),\bB)$ and the group homomorphisms between their $K$-theories \cite{RosenbergSchochet87}), we will see that both $K_0$ and $K_1$ can be obstructions to a unital, simple, purely infinite C*-algebra having a positive solution to Kaplansky's problem.

\subsection{$K$-theoretic obstructions and counterexamples}

We next provide some  counterexamples to pairs of operators satisfying [$\text{KP}_2$] in the setting of purely infinite C*-algebras, where the obstructions  come from topological $K$-theory.
We saw that certain classes of finite classifiable C*-algebras have the property that all pairs of self-adjoint elements satisfy $[\text{KP}_n]$ for all $n \ge 1$.
We now exhibit pairs of projections in certain classifiable purely infinite C*-algebras for which [$\text{KP}_2$] fails.  The obstruction comes from the existence of torsion in the $K_0$-group.
We also give examples showing that $[\text{KP}_n]$ can fail to hold for pairs of unitaries in purely infinite C*-algebras, where the obstruction comes from torsion in the $K_1$-group.

First recall a projection $p$ in a C*-algebra $\bA$ is called \emph{infinite} if it is Murray-von Neumann equivalent to a proper subprojection of itself, and $p$ is \emph{properly infinite} if there are mutually orthogonal projections $p_1,p_2$ in $\bA$ such that $p_1 + p_2 \leq p$ and
\begin{equation}
p \MvN p_1 \MvN p_2.
\end{equation}
A simple C*-algebra is \emph{purely infinite} if every hereditary subalgebra of $\bA$ contains an infinite projection; there are many equivalent definitions, at least in the unital simple setting, see \cite[Proposition 4.1.1]{RordamBook}.
We shall use the following basic facts. See \cite[Proposition 4.1.4]{RordamBook} and \cite[Corollary 6.11.9]{BlackadarKBook} (or appeal to \cite[Theorem 1.7]{Dadarlat95approx} and notice that two projections are unitarily equivalent if and only if, when we identify projections in $\bA \otimes \bK$ with *-homomorphisms $\bC \to A \otimes \bK$, they have the same $KK(\bC,\bA) = K_0(\bA)$-class).

\begin{prop}
Let $\bA$ be a unital, simple, purely infinite C*-algebra. Then
\begin{equation}
K_0(\bA) = \{[p]_0 \mid p \text{ is a non-zero projection in } \bA\}.
\end{equation}
Moreover, if $p,q \in \bA$ are non-trivial projections such that $[p]_0 = [q]_0$, then $p \simeq q$.
\end{prop}


Thus the $K_0$-group can be realized as the $K_0$-classes of non-zero projections in the C*-algebra itself, and two non-trivial projections in the C*-algebra have the same $K_0$-class if and only if they are unitarily equivalent (if and only if they are approximately unitarily equivalent, as close projections are unitarily equivalent).

We now show that there exist pairs of non-unitarily equivalent  projections in Cuntz algebras (more specifically in $\cO_3$) which nonetheless share unitarily equivalent amplifications. Note that $\cO_n \simeq \cO_n \otimes \bM_{n^{\infty}}$, so that $\cO_n$ will always satisfy [$\text{KP}_n$]. Moreover, $\cO_2$ absorbs every unital, simple, separable, nuclear C*-algebra \cite[Theorem 3.8]{KirchbergPhillips00}, so we have that $\cO_2 \otimes \cQ \simeq \cO_2$.  From this and Corollary \ref{cor:Q-stable algebras satisfy KP}, we deduce that any $a, b \in \cO_2$ satisfy \kpn for every $n \in \bN$.

We begin by showing that any unital, simple, purely infinite C*-algebra with projections $p,q$ satisfying $n[p] = n[q]$ but $[p] \neq [q]$ in $K_0$, so that $p$ and $q$  fail to satisfy [$\text{KP}_n$].


\begin{prop}\label{prop: non n-KP}
  Let $\bA$ be a unital, simple, purely infinite C*-algebra.
  Suppose that there exist $n \in \bN$ and $p,q \in \bA$ are non-trivial projections such that $n[p]_0 = n[q]_0$ but $[p]_0 \neq [q]_0$ in $K_0(\bA)$.
  Then
  \begin{equation}
    p^{(n)} \simeq q^{(n)} \text{ in } \bM_n(\bA) \text{ but } p \not\simeq q \text{ in } \bA.
  \end{equation}
  \begin{proof}
    We have that
    \begin{equation}
      [p^{(n)}]_0 = n[p]_0 = n[q]_0 = [q^{(n)}]_0.
    \end{equation}
    As $\bM_n(\bA)$ is unital, simple, and purely infinite, it follows that $p^{(n)} \simeq q^{(n)}$.
    By assumption, we have that $p \not\simeq q$ in $\bA$.
  \end{proof}
\end{prop}


Using the fact that $K_0(\bM_n(\cO_3)) = \bZ_2$ (see \cite{Cuntz81}), we obtain the following.


\begin{cor}
Let $k \in \bN$. There exist two projections in $\bM_k(\cO_3)$ which fail to satisfy [$\text{KP}_2$].
\end{cor}


There is nothing special about [$\text{KP}_2]$ in this regard.   More generally, we can find examples of pairs of operators which do not satisfy [$\text{KP}_n$], even uncountably many. This is because unital, simple, nuclear, purely infinite C*-algebras are in abundance. We can realize any triple $(G_0,g_0,G_1)$, where $G_0,G_1$ are countable abelian groups and $g_0 \in G_0$ is some distinguished element, as the $K$-theory, together with the position of the unit in $K_0$, of such an algebra.
One can also realize $G_0,G_1$ as the K-theory of a stable, separable, simple, nuclear, purely infinite C*-algebra (without keeping track of the unit).
See \cite[Proposition 4.3.3]{RordamBook} (as well as \cite{Rordam95} and \cite{ElliottRordam95}).


\begin{prop}\label{prop: Kirchberg algebra existence}
Let $G_0,G_1$ be countable abelian groups and $g_0 \in G_0$ be some distinguished element.
There is a unital, separable, simple, nuclear, purely infinite C*-algebra $\bA$ satisfying the UCT such that
\begin{equation}
(K_0(\bA),[1_A]_0,K_1(\bA)) \simeq (G_0,g_0,G_1).
\end{equation}
\end{prop}


We can also obtain examples of non-separable C*-algebras (arising from Corona algebras) where Kaplansky's problem admits a negative solution.
Recall that the \emph{multiplier algebra} of a C*-algebra  $\bA$ is the largest unital C*-algebra $\bA$ which contains $\bA$ as an essential ideal. It is in some sense the ``largest'' unitization of $\bA$, and is often thought of as the non-commutative analogue of the Stone--\v{C}ech compactification.
The simplicity of the Corona algebra $\cQ(\bA) := \cM(\bA \otimes \bK)/\bA \otimes \bK$, where $A$ is unital, was characterized by R{\o}rdam in \cite[Theorem 3.2]{Rordam91Ideals}, where it was shown that $\cQ(\bA)$ was simple if and only if $\bA \simeq \bM_n$ for some $n \in \bN$, or $\bA$ was simple and purely infinite.
Additionally, Zhang \cite[Theorem 1.3]{Zhang89} showed that $\cM(\bB)/\bB$ is purely infinite whenever $\bB$ is a $\sigma$-unital simple C*-algebra with real rank zero.
Combining these two results, we have that have $\cQ(\bA)$ is unital, simple, and purely infinite whenever $\bA$ is unital, $\sigma$-unital, simple, and purely infinite.


\begin{cor} \emph{\textbf{[Corona algebras]}}
Fix $n \geq 2$. Let $\bA$ be any unital, separable, simple, purely infinite C*-algebra such that $K_1(\bA)$ has $n$-torsion. Then there exists a pair of projections in $\cQ(\bA)$ which do not satisfy \emph{[$\text{KP}_n$]}.
\begin{proof}
The multiplier algebra $\cM(\bA \otimes \bK)$ has trivial $K$-theory by \cite[Proposition 12.2.1]{BlackadarKBook} (this is often referred to as the Kuiper-Mingo theorem, see \cite[Chapter 16]{WeggeOlsenBook}). Thus, using the 6-term exact sequence in $K$-theory, we have that
\begin{equation}
K_0(\cQ(\bA)) = K_1(\bA) \text{ and } K_1(\cQ(\bA)) = K_0(\bA).
\end{equation}
Now, identifying $K_0(\cQ(\bA))$ with $K_1(\bA)$ and using the fact that $\cQ(\bA)$ is unital, simple, and purely infinite, and that $K_1(\bA)$ has $n$-torsion by hypothesis, it follows that there are distinct projections $p,q \in \cQ(\bA)$ such that $n[p] = n[q]$ but $[p] \neq [q]$ in $K_0(\cQ(\bA))$.
Thus $p^{(n)} \simeq q^{(n)}$ but $p \not\simeq q$ by Proposition \ref{prop: non n-KP}.
\end{proof}
\end{cor}

Now we show that torsion in the $K_1$ group is also an obstruction.
Skoufranis \cite{Skoufranis13} uses the generalised index function $\Gamma$ from \cite{Lin96} to extend the BDF result by demonstrating  that two normal elements with equal spectrum in  a unital, simple, purely infinite C*-algebra  are approximately unitarily equivalent if there are no $K$-theoretic obstructions.


The generalised index is defined as follows: if $\bA$ is a unital C*-algebra and $n \in \bA$ is a normal operator, then we have an embedding of $C(\sigma(n)) \to \bA$ arising from continuous functional calculus: it takes the identity function and maps it to $n$. This homomorphism induces a map at the level of $K_1$, and the generalised index is precisely the map $\Gamma(n): K_1(C(\sigma(n)) \to K_1(\bA)$.


The following is \cite[Theorem 2.16]{Skoufranis13}.

\begin{theorem}\label{thm: Skoufranis unitary orbits in pi}
Let $\bA$ be a unital, simple, purely infinite C*-algebra and let $n_1,n_2 \in \bA$ be normal operators. Suppose
\begin{enumerate}
\item $\sigma(n_1) = \sigma(n_2)$;
\item $\Gamma(n_1) = \Gamma(n_2)$;
\item $n_1$ and $n_2$ have equivalent spectral projections.
\end{enumerate}
Then $n_1 \simeq_a n_2$.
\end{theorem}


There are two things of note when $u$ is a full spectrum unitary. The first is that there will be no non-trivial spectral projections in the C*-algebra, due to the connectedness of $\bT$. The second is that $C(\sigma(u)) = C(\bT)$, and it is well-known that $K_1(C(\bT))$ is generated by the class of the identity function $\id_\bT: z \mapsto z$.
Since $\id_\bT$ generates $C(\bT) \simeq C^*(u)$, and the inclusion $C(\bT) \into \bA$ is given by $\id_\bT \mapsto u$, the generalised index simply records the $K_1$-class of $u$.
As a consequence, two full spectrum unitaries in a unital, simple, purely infinite C*-algebra are approximately unitarily equivalent if and only if they represent the same class in $K_1$.


\begin{prop}
There is a unital simple purely infinite C*-algebra $\bA$ for which each pair of projections in $\bA$ satisfies $[\text{KP}_n]$ for every $n \in \bN$, and yet $\bA$ admits a pair of unitary elements for which $[\text{KP}_2]$ fails.
\begin{proof}
Suppose that $\bB$ is a unital, simple, purely infinite C*-algebra with $K_0(\bB) =0$ and $K_1(\bB) = \bZ_2$, which we know exists by Proposition \ref{prop: Kirchberg algebra existence}.  For a unital, simple, and purely infinite C*-algebra, it well-known that the canonical map $U(\bB)/U^0(\bB) \to K_1(\bB)$ is an isomorphism (see for example \cite[Proposition 8.1.4]{BlackadarKBook}), and so there is a unitary $\tilde{u} \in \bB$ such that $[\tilde{u}] \neq 0$ in $K_1(\bB)$, which simply means that $\tilde{u}$ is not connected to the identity. Note that $\tilde{u}$ must have full spectrum $\sigma(\tilde{u}) = \bT$ as unitary elements whose spectrum is not all of $\bT$  are connected to the identity.

Let $\bA := \bM_2(\bB)$, and note that in fact
\begin{equation}
U(\bB)/U^0(\bB) \simeq U(\bA)/U^0(\bA) \simeq K_1(\bA),
\end{equation}
all via the canonical maps. Consider the unitary elements $u = \tilde{u} \oplus \tilde{u}, v = \tilde{u} \oplus 1_{\bB} \in \bA$.
As $[u] = [\tilde{u}] + [\tilde{u}] = 0$ in $K_1(\bA)$ and $[v] = [\tilde{u}] \neq 0$ in $K_1(\bA)$, it follows that $\Gamma(u) \neq \Gamma(v)$, and so $u \not\simeq_a v$.

Next we show that $u^{(2)} \simeq_a v^{(2)}$ in $\bM_2(\bA)$. In fact, we know that they are approximately unitarily equivalent if and only if they have the same spectrum, $K_1$-class, and all of their spectral projections are equivalent.
Clearly $\sigma(u) = \sigma(v) = \bT$, and lastly, we have that
\begin{equation}
[u \oplus u] = [v \oplus v] = 0
\end{equation}
in $K_1(\bM_4(\bB)) = K_1(\bM_2(\bA))$, meaning that $u^{(2)} \simeq_a v^{(2)}$.

Finally, again, all non-trivial projections in $\bM_n(\bA)$ are unitarily equivalent, for any $n \in \bN$, and so each pair of projections in $\bA$ satisfies the \kpn for any $n \in \bN$.
\end{proof}
\end{prop}



\end{document}